\documentclass[12pt, reqno]{article}
\usepackage{amssymb, mathtools}
\usepackage{latexsym}
\usepackage{amsmath}
\usepackage{amsthm}

\usepackage{epsfig}
\usepackage{verbatim}
\usepackage[all,cmtip]{xy}
\usepackage{enumerate}

\usepackage{authblk}

\usepackage{soul}
\usepackage[usenames,dvipsnames]{color}
\usepackage[colorlinks=true,linkcolor=blue,citecolor=blue]{hyperref}

\newtheorem*{lemma*}{Lemma}

\newtheorem*{theorem*}{Theorem}

\newtheorem{theorem}{Theorem}[section]
\newtheorem{proposition}[theorem]{Proposition}
\newtheorem{lemma}[theorem]{Lemma}

\newtheorem*{claim*}{Claim}

\newtheorem*{conjecture*}{Conjecture}

\theoremstyle{definition}
\newtheorem{definition}[theorem]{Definition}
\newtheorem{example}[theorem]{Example}

\theoremstyle{remark}
\newtheorem{remark}[theorem]{Remark}

\theoremstyle{definition}

\DeclareMathOperator{\ad}{ad}

\newcommand{\diff}{d}

\DeclareMathOperator{\Ker}{Ker}
\DeclareMathOperator{\id}{Id}

\newcommand{\Ran}{\textup{Im}}

\DeclareMathOperator{\ann}{Ann}

\newcommand{\xrightarrowdbl}[2][]{%
	\xrightarrow[#1]{#2}\mathrel{\mkern-14mu}\rightarrow
}

\title{Fibrations and higher products in cohomology}
\author[1]{Alexander Gorokhovsky\thanks{Email: \texttt{alexander.gorokhovsky@colorado.edu}; partially supported by NSF grant DMS-0900968.}}
\author[2]{Zhizhang Xie\thanks{Email: \texttt{xie@math.tamu.edu}; partially supported by NSF grant DMS-1500823. }}
\affil[1]{Department of Mathematics, University of Colorado, Boulder}
\affil[2]{Department of Mathematics, Texas A\&M Univeristy}
\date{}

\begin{document}

\maketitle
\begin{abstract}
This paper is a continuation of \cite{Gorokhovsky:2017aa}. Working 	in the context of commutative differential graded algebras, we study the ideal of the cohomology classes which can be annihilated by fibrations whose fiber has finite homological dimension. In the present paper we identify these classes with   certain higher products in cohomology.

\end{abstract}

\section{Introduction}
This paper is a continuation of \cite{Gorokhovsky:2017aa}. The motivation for these papers  comes from the following question, arising in some geometric situations. Consider a topological space $X$ and a cohomology class $\alpha \in H^\bullet(X; \mathbb{Q})$ of positive degree. Does there exist a fibration $\pi \colon Y \to X$ with a fiber of finite cohomological dimension such that $\pi^* \alpha =0$ (we refer to this by saying that $\alpha$ is annihilated by $\pi$)? It is easy to see that the answer is positive for every class of even degree. Indeed, every such class is a multiple of the Euler class of a sphere fibration and thus is annihilated by a pull back to this fibration. It is easy to see that the set of classes which can be annihilated forms an ideal
in $ H^\bullet(X; \mathbb{Q})$. One would like to obtain some characterization of this ideal.

In \cite{Gorokhovsky:2017aa} this question is considered in the framework of rational homotopy theory.
Namely one replaces topological spaces by commutative differential graded algebras (DGA) and fibrations by algebraic fibrations. In this context, the following two results were proved (see Section \ref{algfib} for the precise definitions).
\begin{enumerate}[$(1)$]
	\item For each commutative DGA  $(A, d)$, there exists an iterated odd algebraic spherical fibration $(TA, d) $ over $(A, d)$ so that  all even cohomology [except dimension zero] vanishes.
	
	\item Let $(B, \diff)$ be a connected commutative DGA  such that $H^{2k}(B)=0$ for all $0 < 2k  \leq 2N$. If $\iota \colon (B, \diff) \to (B\otimes \Lambda V, \diff) $ is an algebraic fibration whose  algebraic fiber has finite cohomological dimension, then the induced map
	\[ \iota_\ast \colon \bigoplus_{i\leq 2N} H^i(B)\to  \bigoplus_{i\leq 2N} H^i(B\otimes \Lambda V)\]
	is injective.
\end{enumerate}

From these results one concludes that a class in cohomology of a commutative DGA can be annihilated by an algebraic fibration with a fiber of finite cohomological dimension if and only if it can be annihilated by an iterated odd algebraic spherical fibration (cf. Proposition \ref{characterisations}).

In the present paper, building on the characterization obtained in \cite{Gorokhovsky:2017aa},  we give another description of the ideal in cohomology of a commutative DGA
$A$,  which consists of elements which can be annihilated by algebraic fibrations with a fiber of finite cohomological dimension.

Given an auxiliary nilpotent differential graded Lie algebra (DGLA) $L$ with central extension, we define a certain higher product operation -- MC product -- in cohomology of $A$ by considering Maurer-Cartan (MC) equation in the DGLA $A \otimes L$. For some choices of $L$ these operations coincide with Massey products.

 For a commutative DGA $A$, let us  denote by $\ann(A)$ the ideal of $H^\bullet(A)$ consisting of the elements which can be annihilated by an algebraic fibration with a fiber of finite cohomological dimension.
The main result of this paper (Theorem \ref{thm:main}) is then the  following:
\begin{theorem}
	
	Let $A$ be a connected commutative DGA. Then the ideal  $\ann(A)$ coincides with the sum of even positive degree cohomology with the set of all MC higher products, i.e., every element in $\ann(A)$ can be written as the sum of an even positive degree cohomology class and a MC higher product. 
\end{theorem}
The identification of $\ann(A)$ with the set of MC higher products allows one to apply results about Maurer-Cartan equation to the study of $\ann(A)$, c.f.  e.g. Lemma \ref{lm:formal}.

This paper is organized as follows. In the section \ref{prelim} we recall some definitions as well as a few results from \cite{Gorokhovsky:2017aa}; we also recall basic facts about Maurer-Cartan equation in DGLA. In the section \ref{sec:mcp} we give a construction of MC higher products. Finally in the section \ref{sec:main} we prove the main results of this paper.

The authors would like to thank D. Sullivan for many   illuminating discussions.

\section{Preliminaries}\label{prelim}
\subsection{Algebraic fibrations}\label{algfib}
In this section we collect some basic definitions as well as a few results from \cite{Gorokhovsky:2017aa} which will be needed in this paper.
\begin{definition}\label{def:algfib}
	
	  An algebraic fibration   is an inclusion of commutative DGAs $(B, \diff) \hookrightarrow (B\otimes \Lambda V, \diff)$ with $V = \oplus_{k\geq 1} V^k$ a graded vector space; moreover,  $V = \bigcup_{n=0}V(n)$, where $V(0) \subseteq V(1) \subseteq V(2) \subseteq \cdots $ is an increasing sequence of graded subspaces of $V$ such that
		\[ d : V(0) \to B \quad \textup{ and } \quad  d: V(n) \to B\otimes \Lambda V(n-1), \quad n\geq 1,\]
		where $\Lambda V$ is the free commutative DGA generated by $V$.
\end{definition}

\begin{definition}
	An \emph{iterated odd algebraic spherical  fibration} over $(B, d)$ is an algebraic fibration $(B, \diff) \hookrightarrow (B\otimes \Lambda V, \diff)$
 such that  $V^k=0$ for $k$ even.     This fibration is called  \emph{finitely iterated odd algebraic spherical fibration} if $\dim V < \infty$.
\end{definition}

\begin{remark}\label{finite}
  Let $(B, \diff) \hookrightarrow (B\otimes \Lambda V, \diff)$  be an iterated odd algebraic spherical  fibration. It is easy to see that for every $x \in B\otimes \Lambda V$ there exists $U \subset V$, $\dim U < \infty$ such that $ B\otimes \Lambda U$ is a differential graded subalgebra of $B\otimes \Lambda V$ and $x \in B\otimes \Lambda U$.
\end{remark}

\begin{theorem}\label{thm:itoddsphere}
	For each commutative DGA  $(A, d)$, there exists an iterated odd algebraic spherical fibration $\tau \colon(A, d) \hookrightarrow  (TA, d) $ over $(A, d)$ such that   $H^{2i}(TA)=0$ for $i>0$.
\end{theorem}
This is proven in   \cite{Gorokhovsky:2017aa} (Theorem 3.3).

\begin{proposition} \label{formality of TA}$(TA,d)$ is formal.
\end{proposition}
This Proposition follows from   Proposition $4.2$ and Corollary $5.6$ of  \cite{Gorokhovsky:2017aa}.

\begin{theorem} \label{thm:inj} Let $(B, \diff)$ be a connected commutative DGA  such that $H^{2k}(B)=0$ for all $0 < 2k  \leq 2N$. If $\iota \colon (B, \diff) \to (B\otimes \Lambda V, \diff) $ is an algebraic fibration whose  algebraic fiber has finite cohomological dimension, then the induced map
	\[ \iota_\ast \colon \bigoplus_{i< 2N} H^i(B)\to  \bigoplus_{i< 2N} H^i(B\otimes \Lambda V)\]
	is injective.
\end{theorem}

This Theorem is also proven in \cite{Gorokhovsky:2017aa} (Theorem 5.8).

\subsection{Maurer-Cartan equation}\label{sec:mc}	
In this section, we recall some basic definitions and constructions of Maurer-Cartan elements in nilpotent differential graded Lie algebras.

Let $\Lambda=\oplus \Lambda^i$ be a nilpotent differential graded Lie algebra (DGLA); $\diff \colon \Lambda^\bullet \to\Lambda^{\bullet+1}$ denotes the differential.
Let $\alpha\in \Lambda^1$ be an element of degree $1$. If we define
\begin{equation*}
F(\alpha) = d\alpha +\frac{1}{2}[\alpha,\alpha],
\end{equation*}
then we have the following Bianchi identity:
\begin{equation*}
\diff F(\alpha) +[\alpha, F(\alpha)]=0.
\end{equation*}
\begin{definition}[Maurer-Cartan element]
An  element $\alpha\in \Lambda^1$ is called a Maurer-Cartan element if
\begin{equation*}
F(\alpha) =\diff \alpha +\frac{1}{2}[\alpha, \alpha]=0.
\end{equation*}
\end{definition}

\begin{definition}[Gauge transformation] For each element $X \in \Lambda^{0}$, we define an automorphism $\exp X : \Lambda \to \Lambda$ by
\begin{equation*}
(\exp X) \cdot \alpha = \alpha - \sum_{i=0}^{\infty} \frac{(\ad
X)^i}{(i+1)!}(\diff X +[\alpha, X])
\end{equation*}
for all $\alpha \in \Lambda$. Such an automorphism is called a gauge transformation of $\Lambda$. The group of gauge transformations of $\Lambda$ will be denoted by  $\exp \Lambda^{0}$.
\end{definition}

The gauge transformation group $\exp \Lambda^{0}$ is a nilpotent group. We also have the following equation:
\begin{equation*}
\diff +(\exp X) \cdot \alpha = e^{\ad X} (\diff+\alpha).
\end{equation*}
It follows that
\begin{equation*}
F((\exp X) \cdot \alpha  )= e^{\ad X} (F(\alpha)).
\end{equation*}
We see that, if $\alpha $ is a Maurer-Cartan element, then $(\exp X) \cdot \alpha $ is also a Maurer-Cartan element.

\begin{definition}
Let $MC(\Lambda)$ be the set of equivalence classes of Maurer-Cartan elements in $\Lambda$ under the action of the gauge transformation group $\exp \Lambda^{0}$.
\end{definition}
Let $f \colon \Lambda_1 \to \Lambda_2$ be a morphism of nilpotent DGLAs. Then it induces a map $f_* \colon MC(\Lambda_1) \to MC(\Lambda_2)$.

\begin{definition}
A DGLA $\Lambda$ is called a filtered DGLA, if there exist subDGLAs $F^i\Lambda$, $i \ge 1$, such that
\begin{enumerate}[(i)]
\item $F^i\Lambda \supseteq F^j\Lambda$ for $i \le j$;
\item  $F^1\Lambda= \Lambda$;
\item $[F^i\Lambda, F^j\Lambda]\subset F^{i+j}\Lambda$.
\end{enumerate}
\end{definition}
A filtration has finite length if    $F^i\Lambda=0$ for $i\gg 0$. Any DGLA with a finite length filtration is nilpotent.
For each filtered DGLA $\Lambda$, the associated graded complex is defined by $\textup{Gr}^\bullet (\Lambda) := \bigoplus_i F^i\Lambda/F^{i+1}\Lambda$, with the differential induced by $d$. For any two filtered DGLAs $\Lambda_1$ and $\Lambda_2$,  a DGLA morphism $f \colon \Lambda_1 \to \Lambda_2$ is called \emph{filtered} if
$f(F^i\Lambda_1) \subset F^i\Lambda_2$ for all $i\geq 1$. Such a morphism induces a morphism of the associated graded complexes.

We recall the following version of  Equivalence Theorem proven in  \cite[Theorem 5.3]{Schlessinger:2012kq}; see also \cite[Theorem 2.4]{MR972343}.
\begin{theorem} \label{thm:GM} Let $\Lambda_1$, $\Lambda_2$ be two DGLAs with finite length filtration.
Let $f \colon \Lambda_1 \to \Lambda_2$ be a filtered morphism   such that the induced map on the associated graded complexes is a quasi-isomorphism. Then the map $f_* \colon MC(\Lambda_1) \to MC(\Lambda_2)$  is a bijection.
\end{theorem}

\section{MC higher products}\label{sec:mcp}
In this section, we introduce the notion of MC higher products, and prove some basic properties.

Let $  0 \to Z  \to \widetilde L \xrightarrow{\pi} L \to 0$
be a central extension of DGLAs concentrated in nonpositive degrees. We denote the differential of $L$ (resp. $\widetilde L$) by $\diff_L$ (resp. $\tilde{\diff}_L$).

\begin{definition}\label{def:ce}
An MC-product data consists of the following:
\begin{enumerate}[$(a)$]

\item a central extension $  0 \to Z  \to \widetilde L \xrightarrow{\pi} L \to 0$
of nilpotent DGLAs concentrated in nonpositive degrees such that $Z\subset \Ker(\tilde{\diff}_L)$ and $\Ran(\tilde{\diff}_L) \cap Z = 0$, where $\Ker(\tilde{\diff}_L)$ is the kernel of $\tilde{\diff}_L$  and $\Ran(\tilde{\diff}_L)$ is the image  of $\Ran(\tilde{\diff}_L)$;

\item an isomorphism of graded vector spaces $Z \cong \mathbb{Q}[q]$, where $q$ is a nonpositive integer and $\mathbb{Q}[q]$ stands for a copy of $\mathbb Q$ placed at degree $q$.
\end{enumerate}
\end{definition}

For every commutative DGA  $(A, d_A)$ and every DGLA $\left(L, \diff_L \right)$, there is natural DGLA $(A \otimes L, \diff)$ defined as follows:
\begin{enumerate}[(1)]
\item $[b_1\otimes w_1, b_2\otimes w_2] \coloneqq (-1)^{|w_1|\cdot |b_2|} (b_1b_2)\otimes [w_1, w_2]$,
\item $\diff (b\otimes w) \coloneqq  (\diff_A b)\otimes w + (-1)^{|b|} b\otimes (\diff_L w)$,
\end{enumerate}
for all $b, b_i\in A$ and $w, w_i\in L$.

\begin{definition} Given an MC-product data
	\[ \lambda = (Z\rightarrowtail \widetilde L \xrightarrowdbl{\pi} L, \diff_L, \tilde \diff_L)\] as in Definition \ref{def:ce}, a \emph{defining system} with respect to $\lambda$ is a  Maurer-Cartan element  in $A^+ \otimes L$, where $A^+ = \oplus_{i>0} A^{i}$.
\end{definition}

Let $\alpha$ be a defining system in $A^+ \otimes L$. Lift $\alpha$ to a degree one element $\tilde{\alpha}$ in   $ A^+ \otimes \widetilde{L}$.  We have  $(\id \otimes \pi) (F( \tilde{\alpha}))= F(\alpha) =0$ in $ A^+ \otimes L$.
It follows that $F(\tilde{\alpha}) \in \Ker (\id \otimes \pi) =A^+ \otimes Z \cong A^+[q]$, where $ A^+[q]$ is the $q$-shift of the complex $A^+$, that is, $ (A^+[q])^k = (A^+)^{k+q}$.

\begin{lemma}With the same notation as above, we have
\begin{equation*}
\diff F(\tilde{\alpha})=0.
\end{equation*}
\end{lemma}
\begin{proof}
	$F(\tilde \alpha)$ is a central element, since $F(\tilde{\alpha}) \in \Ker (\id \otimes \pi)$.
It follows that
\begin{equation*}  \diff F(\tilde \alpha)  = - [\tilde \alpha, F(\tilde \alpha)] = 0. \end{equation*}
\end{proof}
\begin{lemma}
The cohomology class of $F(\tilde{\alpha})$ depends only on the defining system $\alpha$.
\end{lemma}
\begin{proof}
Suppose  $\tilde \alpha_1\in A^+ \otimes \widetilde{L}$ is another lift of $\alpha$. Then $ \tilde \alpha_1 - \alpha = b\otimes z$, for some $b\in A^+$ and $ z \in Z$. Then we have
 \begin{align*}
  F(\tilde \alpha_1) & = \diff\tilde \alpha_1 + \frac{1}{2}[\tilde \alpha_1, \tilde \alpha_1] \\
& = \diff \tilde \alpha + \diff b \otimes z  + \frac{1}{2}[\tilde \alpha + b \otimes z, \tilde \alpha + b\otimes z] \\
& = \diff  \tilde \alpha + \frac{1}{2}[\tilde \alpha , \tilde \alpha] + \diff b \otimes z = F(\tilde \alpha) + \diff b \otimes z.
 \end{align*}
This finishes the proof.
\end{proof}

Therefore,  for each defining system $\alpha\in A^+\otimes L$, there is a well-defined cohomology class $m(\alpha)\coloneqq [F(\tilde{\alpha})]\in H^{2-q}(A)$.
\begin{definition}
We call $m(\alpha)$ the MC higher product
of the defining system $\alpha \in A^+\otimes L$.
\end{definition}
The following lemma states that MC higher products are invariant under gauge transformations.
\begin{lemma}\label{gauge} For every $X\in (A^+\otimes L)^{0}$, we have
\begin{equation*} m(\exp X \cdot \alpha)= m(\alpha). \end{equation*}
\end{lemma}
 \begin{proof}
 Let $\widetilde{X}\in \widetilde{L} $ be a lift of $X$. Clearly, $\exp \widetilde{X} \cdot \tilde{\alpha}$ is a lift of $\exp X \cdot \alpha$. Then we have
 \[ F(\exp \widetilde{X} \cdot \tilde{\alpha}) = e^{\ad \widetilde{X}} F(\tilde{\alpha})= F(\tilde{\alpha}). \] The second equality follows from the fact that $F(\tilde{\alpha})$ is central.
 \end{proof}

\begin{definition} Let $\lambda = (Z\rightarrowtail \widetilde L \xrightarrowdbl{\pi} L, \diff_L, \tilde \diff_L)$ be an MC-product data.
 For a  commutative DGA $A$, we denote by  $MP_{\lambda}(A)$ the subset of $H^\bullet(A)$ consisting of  MC higher products  $m(\alpha)$, where $\alpha$ runs through all defining system with respect to $\lambda$.
\end{definition}

It is clear that any morphism
$f \colon A \to B$ of commutative DGAs induces a map $f_* \colon MP_\lambda(A)  \to MP_\lambda(B)$.
\begin{proposition}\label{gm}
If  $f \colon A \to B$ is a quasi-isomorphism of connected commutative  DGAs, then $f_* \colon MP_\lambda(A)  \to MP_\lambda(B)$ is a bijection.
\end{proposition}
\begin{proof}
The injectivity of $f_*$ follows from the injectivity of the induced map $H^\bullet(A) \to H^\bullet(B)$. To prove surjectivity
consider    the filtration of $L$ by the central series $\{L^i\}$, where  $L^i=[L, L^{i-1}]$ and $L^1=L$. We filter $A^+ \otimes L$ by $F^i(A^+ \otimes L)=A^+ \otimes L^i$ and   $B^+\otimes L$ by $F^i(B^+ \otimes L)=B^+ \otimes L^i$.
Now surjectivity follows   from  Theorem $\ref{thm:GM}$ and Lemma \ref{gauge}.
\end{proof}

\begin{example}\label{ex:Massey} Let $\{a_i\}_{1\leq i \leq n}$ be a collection of elements  of a commutative DGA $A$ such that $\diff a_i=0$.

Consider the Lie algebra $L$ generated by $\varepsilon_i$ with  the relations $[\varepsilon_i, \varepsilon_j]=0$ unless $|i-j|=1$, and  commutators of order $\ge n$ are $0$. Here we set $|\varepsilon_i|=1-|a_i|$. We equip $L$ with the zero differential.   Let $\widetilde{L} = L \oplus \mathbb{Q}\eta$, where $\mathbb Q\eta$ is the one dimensional Lie algebra generated by $\eta = [\varepsilon_1, [\varepsilon_2,[\ldots,[\varepsilon_{n-1}, \varepsilon_n]\ldots]$.  Then $\lambda = (\mathbb Q\eta \rightarrowtail \widetilde L \twoheadrightarrow L, 0, 0)$ is an MC-product data.  Consider an element $\alpha \in A\otimes L$ of the form
\begin{equation*}  \alpha = \sum_{i=1}^{n}a_i \otimes \varepsilon_i + \sum_{\substack{1\leq i<j \leq n \\ (i, j)\neq (1, n)} } a_{(i, j)} \otimes  [\varepsilon_i, [\varepsilon_{i+1},[\ldots,[\varepsilon_{j-1}, \varepsilon_{j}]\ldots]  \end{equation*}
where $ a_{(i, j)}\in A$. Suppose  $\alpha$ is a defining system with respect to $\lambda$. It is not difficult to see that the equation
\begin{equation*}  F(\alpha)  = \diff \alpha + \frac{1}{2}[\alpha, \alpha] = 0 \end{equation*}
gives precisely the standard relations in Massey product $\langle a_1, a_2, \ldots, a_n\rangle$,  cf. \cite{MR0098366}. Moreover, if $\tilde \alpha\in A^+\otimes \widetilde L$ is a lift of $\alpha$, then $F(\tilde \alpha)$ is precisely a Massey product of  $\{a_1, a_2, \ldots, a_n\}$.
\end{example}
\begin{remark} The close connection between defining systems and twisting cochains has been noticed in \cite{MR0394720}, Remark 5.5.
\end{remark}
The following lemma will be useful later.
 \begin{lemma}\label{lm:formal} Let $(B, d)$ be a formal commutative DGA. Then for any MC product data $\lambda = (Z\rightarrowtail \widetilde L \xrightarrowdbl{\pi} L, \diff_L, \tilde \diff_L)$, every element in $MP_\lambda(B)$ is a finite sum of  decomposable elements in $H^\bullet(B)$.
 \end{lemma}
 \begin{proof} By Proposition \ref{gm}, without loss of generality, we can assume that $B$ has a zero differential. Choose a splitting DGLA morphism $p \colon \widetilde{L} \to Z$ of the inclusion $\iota: Z\rightarrowtail \widetilde{L}$ such that $p \circ \tilde \diff_L =0$, where $\tilde \diff_L$ is the differential on $\widetilde L$. For example, let us define $p|_{\Ran(\tilde \diff_L)} = 0 $ and $p|_Z = \textup{Id}|_Z$. This is well-defined on $\Ran(\tilde \diff_L)+Z$, since $\Ran (\tilde \diff_L) \cap Z = 0$. Now take an arbitrary linear extension of  $p$ from $\Ran(\tilde \diff_L)+Z$ to $\widetilde L$.
 	
 	Suppose  $\alpha \in B^+ \otimes L$ is a defining system with respect to $\lambda$, and
 $\tilde{\alpha}\in B^+\otimes \widetilde L$ is a lift of $\alpha$. Since the differential on $B$ is zero, we have
 	\begin{equation*} m(\alpha) = [F(\tilde{\alpha})]= [p(F(\tilde{\alpha}))]= \frac{1}{2} [p([\tilde{\alpha}, \tilde{\alpha}])]. \end{equation*}
 	The last expression
 	is a finite sum of decomposable elements in $B^+$. This completes the proof.
 \end{proof}

\begin{remark}\label{GF}
Constructions of this section are special cases of the following. Assume that we have the following data:
\begin{itemize}
\item $(A, d_A)$ -- a connected commutative DGA (in degrees $\ge 0$);
\item $(L, d_L)$ -- a DGLA concentrated in degrees $\le 0$;
\item $\mu$ -- an MC element in $A^+ \otimes L$.
\end{itemize}
Consider $C^\bullet_{Lie}(L) = (\Lambda^\bullet L^*, \delta_{Lie}+d_L) $ -- the standard Lie cohomology complex of DGLA $L$  with the differentials
being the Lie cohomology coboundary $\delta_{Lie}$ (up to a sign) and $d_L$.
Let $\eta \in C^\bullet_{Lie}(L)$. Extend $\eta$ to a map
\[
\eta \colon C^\bullet_{Lie}(A \otimes L) \to A
\]
by $ \eta(a_1\otimes l_1, \ldots, a_k\otimes l_k) = (-1)^\varepsilon a_1\ldots a_k \, \eta(l_1, \ldots, l_k),$ where
\[
\varepsilon =|\eta|\sum |a_i|+\sum_{i>j} |a_i|(|l_j|+1) \text{ and } |\eta|=\sum (|l_i|+1).
\]
Then the map
$
\gamma_\mu \colon C^\bullet_{Lie}(L) \to A
$
  given by
\[
\gamma_\mu(\eta) :=   \frac{1}{k!} \eta(\mu, \mu, \ldots, \mu), \ \text{ if } \eta \in \Lambda^kL^*,
\]
 is a morphism of complexes, i.e.
\[
\diff_A(\gamma_\mu(\eta)) = \gamma_\mu((\delta_{Lie} + \diff_L) \eta).
\]
As a consequence,   we obtain a \emph{characteristic map}
\begin{equation}
\gamma_\mu \colon H^\bullet_{Lie}(L) \to H^\bullet(A).
\end{equation}
This map depends only on the gauge equivalence class of $\mu$.
Indeed, let $X \in (A^+ \otimes L)^0$ and let $\mu_0$ and $\mu_1$ be two MC elements in $A^+ \otimes L$ with $\mu_1 = (\exp X) \cdot \mu_0$.
Then we have
\[
\gamma_{\mu_1}(\eta) -\gamma_{\mu_0}(\eta) =(-1)^{|\eta|}\big( \diff_A H(\eta) - H\big( ( \delta_{Lie} + \diff_L ) \eta\big)\big)
\]
where
\[
H(\eta) = \frac{1}{(k-1)!}\int_0^1 \eta (X, \mu_t, \mu_t, \dots, \mu_t) dt, \textup{ with } \ \mu_t : = (\exp tX)\cdot \mu_0.
\]

\begin{example}
 Suppose we are given a commutative DGA $(A, \diff_A)$,  an MC-product data $ \lambda = (Z\rightarrowtail \widetilde L \xrightarrowdbl{\pi} L, \diff_L, \tilde \diff_L)$ and a Maurer-Cartan element $\mu \in A^+\otimes L$.   The central extension $Z\rightarrowtail \widetilde L \xrightarrowdbl{\pi} L$ defines a class in $H^\bullet_{Lie}(L)$ (its degree is the degree of the center shifted by $2$). The image of this class (under the characteristic map $\gamma_{\mu}$ above) in $H^\bullet(A)$ is precisely the MC higher product $m(\mu)$.
\end{example}

\begin{example} The data described in the beginning of Reamrk \ref{GF} arises naturally, for example, in the following situation.
Let $(A, d_A)$ be a commutative DGA   and $(A\otimes \Lambda V, D)$ be an algebraic fibration with $\Lambda V$ of finite type. Denote by $d_V$ the differential on $\Lambda V$. Let $L$ be a DGLA of filtered derivations of $\Lambda V$,  truncated  at degree $0$, and its  differential  of degree $+1$ given by $d_L:=[d_V, \cdot]$.
Then  $D=d_A \otimes 1+ 1\otimes d_V + \mu$, where $\mu$ can be thought of as a Maurer-Cartan element of DGLA $A^+ \otimes L $, cf. \cite[Theorem 9.2]{Schlessinger:2012kq}.
In this situation the gauge equivalence of MC elements corresponds to an isomorphism of fibrations.
\end{example}
\end{remark}

\section{Main theorem}\label{sec:main}

\begin{definition}
	Let $A$ be a connected commutative DGA. $\ann(A)$ denotes the set of elements $a\in H^\bullet(A)$ for which there exists an algebraic fibration whose fiber has finite cohomological dimension
	$ \iota\colon (A, \diff)  \hookrightarrow (A\otimes \Lambda V, \diff )$
	such that $\iota_\ast(a) = 0$ in $H^\bullet(A\otimes \Lambda V)$.
\end{definition}


It is clear that $\ann(A)$ is an ideal of $H^\bullet(A)$. For the next Proposition recall the iterated odd algebraic spherical   fibration $\tau \colon(A, \diff) \hookrightarrow  (TA, \diff) $ from Theorem \ref{thm:itoddsphere}.
\begin{proposition}\label{characterisations} Let $a \in H^\bullet(A)$. Then the following are equivalent:
 \begin{enumerate}[$(1)$]
   \item \label{a} $a \in \ann(A)$.
   \item \label{b} $a \in \Ker \tau_*$.
   \item \label{c} There exists a finitely iterated odd algebraic spherical fibration $\iota \colon (A, \diff) \hookrightarrow (A\otimes \Lambda V, \diff)$ such that $\iota_*(a)=0$.
 \end{enumerate}
\end{proposition}
\begin{proof}
 Assume $a \in H^\bullet(A)$ and $\iota \colon (A, d)  \hookrightarrow (A\otimes \Lambda V, d)$ is an algebraic fibration such that $\iota_*(a)=0$. Consider the pushout fibration $\varphi \colon TA \hookrightarrow TA \otimes \Lambda V$.  Then $(\varphi\circ \tau)_*(a)=0$. Since by Theorem \ref{thm:inj} $\varphi_*$ is injective, we obtain that $\tau_*(a)=0$. This proves that \eqref{a} implies \eqref{b}. The implication \eqref{b} $\Rightarrow$ \eqref{c} follows from Remark \ref{finite}. Finally, \eqref{a} is an obvious consequence of \eqref{c}.
\end{proof}

The main result of this paper is the following characterization of $\ann(A)$ in terms of Maurer-Cartan higher products.
\begin{theorem}\label{thm:main}

 Let $A$ be a connected commutative DGA. Then
\begin{equation*}  \ann(A) = \bigoplus_{k>0}H^{2k}(A) + \bigcup_{\lambda} MP_\lambda(A)  \end{equation*}
where $\lambda$ runs through all MC-product data.
\end{theorem}
\begin{proof}
We first prove the inclusion $\bigoplus_{k>0}H^{2k}(A) + \bigcup_{\lambda} MP_\lambda(A) \subseteq \ann (A)$. By Proposition \ref{characterisations} it suffices to show that $\tau_* \left( \bigoplus_{k>0}H^{2k}(A) + \bigcup_{\lambda} MP_\lambda(A)  \right)=0$. Since  $H^{2k}(TA) = 0$ for all $k>0$, $\tau_* \left( \bigoplus_{k>0}H^{2k}(A)\right)=0$. Note also that the product in $H^\bullet(TA)$ is zero.  Since $TA$ is formal (Proposition \ref{formality of TA}),  Lemma \ref{lm:formal} implies that
$MP_\lambda(TA)=0$ for any MC product data $\lambda$. But $\tau_\ast MP_\lambda (A) \subset MP_\lambda (TA)$.  So the statement is proved.

  Now we prove the inclusion
\begin{equation*} \ann(A) \subseteq \bigoplus_{k>0}H^{2k}(A) + \bigcup_{\lambda} MP_\lambda(A).  \end{equation*}
Let $a$ be an odd degree element in $\ann(A)$. By  Proposition \ref{characterisations} there exists a finitely iterated odd algebraic spherical fibration   $\iota\colon A \hookrightarrow A\otimes \Lambda V$ such that $\iota_*(a)=0$. We can write $A\otimes \Lambda V$ as $A[x_1, \cdots, x_n]$,  where $x_i$ are variables of odd degree such that $
 \diff x_1 \in A$ and $\diff x_i\in A[x_1, \cdots, x_{i-1}].$

 We shall prove that $a\in   \bigcup_{\lambda} MP_\lambda(A)$ by induction on the number of variables. When $n=1$, by Gysin sequence the kernel of $\iota_\ast\colon H^{\bullet}(A) \to H^{\bullet}(A[x_1])$ is precisely $e \cdot H^\bullet(A)$, where  $e = \diff x_1$ is the Euler class of the algebraic spherical fibration $A\hookrightarrow A[x_1]$. It follows that   $a = b \cdot e$ with $b\in H^\bullet(A)$. The ordinary cup product is   a Maurer-Cartan higher product (see   Example \ref{ex:Massey}), so the case where $n=1$ is proved. The proof of the induction step is contained in Proposition $\ref{prop:ind}$ below. This finishes the proof.

\end{proof}
%
%
%
%
%
%
%
%
%
%
%
%
%

The rest of this section  is devoted to completing the proof of Theorem $\ref{thm:main}$. We will prove the following key proposition.

\begin{proposition}\label{prop:ind}
 Suppose $i \colon (A, \diff)\to (A[x], \diff)$ is an odd algebraic spherical fibration with $\diff x= e \in A$.  Given an odd degree element $c \in H^\bullet(A)$, if  $i_*(c) \in MP_\lambda(A[x])$ for some MC product data $\lambda$, then there exists an MC product data $\zeta =\zeta(\lambda, c)$   such that
\begin{equation*}  c \in MP_\zeta(A)
.   \end{equation*}
\end{proposition}

\begin{proof}
Denote the MC-product data $ \lambda $ by $\lambda =  (Z\rightarrowtail \widetilde L \xrightarrowdbl{\pi} L, \diff_L, \tilde \diff_L)$. Let $x\omega +\theta \in (A[x])^+ \otimes L$ be the defining system that produces $i_*(c)$, where $\omega \in A\otimes L$,  $\theta \in A^+\otimes L$. Suppose
$\alpha \in A\otimes \widetilde L$ and $\beta \in A^+ \otimes \widetilde{L}$ are lifts of $\omega$ and $\theta$ respectively. Suppose the degree of $x$ is $n$. We have $|\alpha| = 1-n$ and $|\beta| = 1$.  It follows that
\begin{align*}
F(x\alpha +\beta) &= \diff(x\alpha+ \beta) + \frac{1}{2}[x\alpha +\beta, x\alpha +\beta] \\
& = e\alpha - x (\diff \alpha) + \diff \beta + \frac{1}{2}\big(x[\alpha, \beta] + (-1)^{|\beta||x|} x [\beta, \alpha] + [\beta, \beta]\big) \\
&  = e\alpha +\diff \beta + \frac{1}{2}[\beta, \beta] + x ([\alpha, \beta] - \diff \alpha).
\end{align*}
By assumption, $F(x\alpha + \beta)$ is cohomologous to $i_\ast (c)$, thus
\begin{equation*}    F(x\alpha + \beta) = cz + \diff(xu + v) \end{equation*}
for $u, v\in  A\otimes Z \subset A\otimes \widetilde L$, and $z \in Z$.
Replacing $\alpha$ by $\alpha-u$ -- another lifting of $\alpha$, and $c$ by $c+dv$ (where we identify $A\otimes Z$ with $A$ shifted by $q$) -- another representative of the same cohomology class,
we can assume that
 \begin{equation*}    F(x\alpha+ \beta) = cz, \end{equation*}
which can be rewritten as
\begin{equation}\label{MC1eq}   -\diff \alpha +[\alpha,\beta] = 0, \end{equation}
\begin{equation}\label{MC2eq}   e\alpha +\diff \beta + \frac{1}{2}[\beta, \beta] = cz.  \end{equation}

We now construct an auxiliary DGLA $N$ and a central extension of DGLAs
$Z \rightarrowtail \widetilde{N} \twoheadrightarrow N$.

Let us first fix some notation. Let $\varepsilon$ be a variable of degree $n$ such that $\varepsilon^2=0$.
We define
\begin{equation*} \widetilde N = \mathbb{Q} \eta \ltimes (\widetilde{L}[\varepsilon]),\end{equation*} where $\mathbb{Q} \eta$ is the one-dimensional Lie algebra generated by $\eta$ of degree $(-n)$ and  acts on $\widetilde{L}[\varepsilon]$ by $\frac{\diff}{\diff \varepsilon}$. More explicitly,
\begin{equation*} \widetilde N=
\widetilde{L}\oplus \widetilde{L} \varepsilon \oplus \mathbb{Q}\eta, \end{equation*} and the following relations hold:
\begin{equation*} [s\varepsilon, t\varepsilon]=0, [\eta, \eta]=0, [\eta, t] = 0,  \end{equation*}
\begin{equation*}  [\eta, s\varepsilon] =   (-1)^{|s|}s, \end{equation*}
 \begin{equation*} [s, t\varepsilon]= [s, t]\varepsilon, \end{equation*}
  for all $s, t\in \widetilde{L}$. The differential $\tilde \diff_N$ on $\widetilde N$ is defined by
  \begin{equation*}
  \tilde{\diff}_N (\ell) =\tilde{\diff}_L (\ell), \ \tilde{\diff}_N (\eta) =0, \ \tilde{\diff}_N (\ell\varepsilon) =\tilde{\diff}_L (\ell) \varepsilon,
  \end{equation*}
  where $\tilde\diff_L$ is the differential on $\widetilde L$. Consider the subDGLA
$Z \subset \widetilde L \subset \widetilde N$. It is clear that $Z$  is central  in $\widetilde N$, and $Z\subset \Ker(\tilde{\diff}_N)$. We define $N: = \widetilde{N}/Z$. Moreover, since $\widetilde L$ is nilpotent, both $\widetilde N$ and $N$ are nilpotent. However, in general, both $\widetilde N$ and $N$  may not be  concentrated in nonpositive degrees.

\begin{lemma}\label{curvg} Set $\mu = \alpha \varepsilon+ \beta+ e\eta \in (A \otimes \widetilde{N})^1 $. Then $F(\mu) = cz$.
\end{lemma}
There exist   $\ell_0 \in \widetilde{L}^{1-n}$ and $\bar{a} \in A^+ \otimes \widetilde{L}^{\le -n}$ such that $\alpha = 1\otimes \ell_0 + \bar{\alpha}$.
\begin{lemma} \label{dl} $\tilde{\diff}_L(\ell_0)=0.$
\end{lemma}
\begin{proof}
From \eqref{MC1eq}, we have
\[
1\otimes \tilde{\diff}_L(\ell_0) =-\diff \bar{\alpha}+[\alpha, \beta].
\]
But the right hand side is in $A^+ \otimes \widetilde{L}$ and the statement follows.
\end{proof}

We note that the degree of $\ell_0$ is even. At the same time the degree of the MC higher product $c$ is $2-q$, where   $q$ is the degree of the  elements of $Z$. Since by our assumption $|c|$ is odd, $q$ is odd as well. As a consequence, $\ell_0 \notin Z$.

Consider now a DGLA $\widetilde N'$ which coincides with $\widetilde N$ as a graded Lie algebra; the differential $\tilde{\diff}_{N'}$ is given by
\[
\tilde{\diff}_{N'} = \tilde{\diff}_{N} +[\ell_0\varepsilon, \cdot].
\]
Lemma \ref{dl} together with the obvious identity $[\ell_0\varepsilon, \ell_0\varepsilon]=0$ imply that  $\widetilde N'$ is a DGLA.

In these terms,  Lemma \ref{curvg} can be rewritten as
\begin{lemma}Set $\bar{\mu}= \bar{\alpha} \varepsilon+\beta+ e\eta \in (A^+ \otimes \widetilde N')^1 $. Then $ F(\bar{\mu}) = cz.$

\end{lemma}

We now define a subDGLA $\widetilde M\subset \widetilde N' $ as a truncation of $\widetilde N'$ at degree $0$, i.e.
\[
\widetilde M^i=
\begin{cases}
0 &\text{ if } i>0,\\
\Ker (\tilde{\diff}_{N'}) \cap  (\widetilde N')^0 &\text{ if } i=0,\\
 (\widetilde N')^i &\text{ if } i<0.
\end{cases}
\]
We denote the restriction of $\tilde \diff_{N'}$ on $\widetilde M$ by  $\tilde \diff_M$ .
\begin{lemma} $\bar{\mu} \in (A^+ \otimes \widetilde M)^1$.
\end{lemma}
\begin{proof}
It is clear from the construction that $\bar{\mu} \in A^+ \otimes (\widetilde N')^{\le 0}$.
Write $\bar \mu=\bar \mu_0+\bar \mu_-$, where  $\bar \mu_0 \in A^+ \otimes (\widetilde N')^{0}$ and $\bar \mu_-\in A^+ \otimes (\widetilde N')^{<0}$. We need to verify that
$\bar \mu_0 \in A^+\otimes (\widetilde M)^0$. We have
$F(\bar \mu) = (1\otimes \tilde{\diff}_{N'}) \bar \mu_0 +R$, with $(1\otimes \tilde{\diff}_{N'}) \bar \mu_0 \in A^+ \otimes (\widetilde N')^{1}$ and  $R \in  A^+ \otimes (\widetilde N')^{0}$. It follows that $(1\otimes \tilde{\diff}_{N'}) \bar \mu_0 =0$, hence $  \bar \mu_0 \in A^+ \otimes \Ker \tilde{\diff}_{N'}$.
\end{proof}

Note that $Z$ is also a central subDGLA of $\widetilde M$.
\begin{lemma} $Z\subset \Ker(\tilde \diff_M)$ and $\Ran(\tilde \diff_M) \cap Z = 0$.
\end{lemma}
\begin{proof}
The first statement is clear. For the second statement assume that, to the contrary, $ \tilde{\diff}_M(w) \in Z$ is nonzero. Observe first that  elements in $Z$ have odd degrees, so if $ \tilde{\diff}_M(w) \in Z$, then $|w|$ is even.
It follows that $w= s+ t\varepsilon$, where $s$, $t \in \widetilde L$.  Then
$  \tilde{\diff}_{M} (s + t\varepsilon) =\tilde{\diff}_L(s) + r \varepsilon$ for some $r \in \widetilde L$.  It follows that
$\tilde{\diff}_L(s) \in Z$ and $r=0$. This is possible only if $\tilde{\diff}_L(s)=0$. Therefore,  $\tilde{\diff}_M(w)=0$. This finishes the proof.
\end{proof}

Now consider the central extension of nilpotent DGLAs $Z \rightarrowtail \widetilde M \twoheadrightarrow M$, where $M=\widetilde M/Z$ with  the differential $\diff_M$ induced by $\tilde \diff_M$. We define $\zeta$ to be the MC product data $( Z \rightarrowtail \widetilde M \twoheadrightarrow M, \tilde \diff_{M}, \diff_M)$. The lemmas above showed that $\bar \mu$ is a defining system with respect to $\zeta$, and its associated MC higher product is $c$. This finishes the proof.

\end{proof}


\end{document}